\documentclass[oneside,english]{amsart}
\usepackage{amsthm}
\usepackage{amssymb}
\usepackage{esint}

\makeatletter
\numberwithin{equation}{section}
\numberwithin{figure}{section}
\theoremstyle{plain}
\newtheorem{thm}{\protect\theoremname}
  \theoremstyle{plain}
  \newtheorem{prop}[thm]{\protect\propositionname}
  \theoremstyle{plain}
  \newtheorem{cor}[thm]{\protect\corollaryname}
  \theoremstyle{plain}
  \newtheorem{lem}[thm]{\protect\lemmaname}

\makeatother

\numberwithin{thm}{section}

\usepackage{babel}
  \providecommand{\corollaryname}{Corollary}
  \providecommand{\lemmaname}{Lemma}
  \providecommand{\propositionname}{Proposition}
\providecommand{\theoremname}{Theorem}

\begin{document}

\title{Factorizations of Analytic Self-Maps of the Upper Half-Plane}

\author{Hari Bercovici and Dan Timotin}

\thanks{HB was supported in part by grants from the National Science Foundation.
DT was supported in part by a grant of the Romanian National Authority for Scientific
 Research, CNCS--UEFISCDI, project number PN-II-ID-PCE-2011-3-0119.}

\address{HB: Department of Mathematics, Indiana University, Bloomington, IN
47405, USA}

\email{bercovic@indiana.edu}

\address{DT: Simion Stoilow Institute of Mathematics of the Romanian Academy,
PO Box 1-764, Bucharest 014700, Romania}

\email{Dan.Timotin@imar.ro}

\begin{abstract}
We extend a factorization due to Kre\u\i n to arbitrary analytic
functions from the upper half-plane to itself. The factorization represents
every such function as a product of fractional linear factors times
a function which, generally, has fewer zeros and singularities than
the original one.  The reult is used to construct functions with given zeros and poles on the real line.
\end{abstract}
\maketitle

\section{Introduction}

Denote by $\mathbb{C}^{+}=\{x+iy:x\in\mathbb{R},y>0\}$ the upper
half of the complex plane. We consider the multiplicative structure
of the set $\mathcal{H}$ of analytic functions $f\colon\mathbb{C}^{+}\to\overline{\mathbb{C}^{+}}$.
We start with a classical result of M.
G. Kre\u\i n \cite[Theorem 27.2.1]{levin}. Assume that $k\in\mathcal{H}$
is the restriction of a meromorphic function defined on $\mathbb{C}$
with the property that $k(\overline{z})=\overline{k(z)}$ and, in
addition, $k$ has infinitely many positive and infinitely many negative
zeros. Denote by $\{a_{n})_{n=1}^{\infty}$ the zeros of $k$, all
of which are real and simple, and by $\{b_{n}\}_{n=1}^{\infty}$ all
the poles of $k$, all of which are also real and simple, arranged
so that the interval $(b_{n},a_{n})$ contains no zeros or poles.
Then Kre\u\i n proved that there exist positive constants $\{c_{n}\}_{n=1}^{\infty}$
such that
\begin{equation}
k(z)=\prod_{n=1}^{\infty}c_{n}\frac{z-a_{n}}{z-b_{n}}.\label{eq:k-product}
\end{equation}
The constants $c_{n}$ are needed for convergence of the product,
and they can be taken to be equal to $b_{n}/a_{n}$ for all but one
or two values of $n$. This result was studied in further detail by
Chalendar, Gorkin, and Partington \cite{chal-et-al}, including the
cases where there are, for instance, only finitely many negative zeros.
Furthermore, \cite{chal-et-al} addresses the more general situation
where $f$ is not meromorphic, but is allowed a finite number of essential
singularities on the real line. However the results are not as complete
in this more general situation.

Our main result states that any function $f\in\mathcal{H}$ can be
written as
\[
f=kg,
\]
where $k$ is a product of the form (\ref{eq:k-product}), and $g\in\mathcal{H}$
has the following additional property: if $I\subset\mathbb{R}$ is
any interval such that $g$ extends continuously to a real-valued
function on $I$, then the extension is positive on $I$. This factorization
is of interest only for functions $f$ which do extend to a real function
on a nonempty open subset $\Omega$ of $\mathbb{R}$. When the complement
of this set $\Omega$ has linear measure equal to zero, the factor
$g$ is a positive constant, and this extends Kre\u\i n's theorem
as well as the factorization results in \cite{chal-et-al}.  In particular, we se in Section 4 aunder what conditions we can prescribe the real zeros and poles of a function in $\mathcal{H}$.

Our factorization is closely related to results of Aronszajn and Donoghue
\cite{aron-don}. These authors also consider the multiplicative
structure of $\mathcal{H}$ starting with the observation that $\log f\in\mathcal{H}$
provided that $f\in\mathcal{H}\setminus\{0\}$. A result essentially
equivalent to Kre\u\i n's theorem is stated in \cite[page 331]{aron-don}.
The emphasis in these works is however not on factorization into linear
fractional factors. We also refer to Donoghue's book \cite{donog}
for information about the class $\mathcal{H}$, and Section 2
of \cite{gest-tsek} for a brisk review of the basic results concerning
this class.

\section{Kre\u\i n Products}

We now describe in more detail the products $k$ needed in our results.
Given a proper open interval $J\subset\mathbb{R}\cup\{\infty\}$, there exists a unique conformal automorphism $p_{J}\in\mathcal{H}$
such that $p_{J}$ is negative precisely on $J$, and $|p_{J}(i)|=1$.
Explicitly, if $J=(b,a)$ with $b<a$ is a finite interval, then
\[
p_{J}(z)=\frac{|i-b|}{|i-a|}\cdot\frac{z-a}{z-b}.
\]
If $J=(-\infty,a),$ then
\[
p_{J}(z)=\frac{z-a}{|i-a|}.
\]
If $J$ is the complement of a closed interval $\overline{J'}$, where
$J'$ is of the preceding two types, then
\[
p_{J}=-\frac{1}{p_{J'}}.
\]
We also agree that $p_{\varnothing}(z)=1$ and $p_{\mathbb{R}\cup\{\infty\}}(z)=-1$.

In order to simplify the notation of intervals in $\mathbb R\cup\{\infty\}$, we set
\[
(b,a)=(b,\infty)\cup\{\infty\}\cup(-\infty,a)
\]
when $b>a$.

Consider now an arbitrary open subset $O\subset\mathbb{R}\cup\{\infty\}$,
and write it as a union of pairwise disjoint open intervals
\[
O=\bigcup_{0\le n<N}J_{n},\quad J_n=(b_n,a_n),\quad0\le n<N,
\]
where $N\in\{0,1,\dots,\infty\}$. Denote by $X$ the closure of the set $\{b_n:0\le n<N\}$  in $\mathbb{R}\cup\{\infty\}$.  The function
\[
k_{O}(z)=\prod_{0\le n<N}p_{J_{n}}
\]
 is called the \emph{Kre\u\i n product }associated with the set $O$.

We collect in the following statement the basic properties of Kre\u\i n products.  The proofs are similar to the ones available in the classical case where $\sup_n b_n=+\infty$ and $\inf_n b_n=-\infty$.  We sketch these arguments for the reader's convenience.

\begin{prop}
\label{prop:krein-prod-properties}Consider an open set $O=\bigcup_{0\le n<N}(b_n,a_n)\subset\mathbb R\cup\{\infty\}$, where the intervals $(b_n,a_n)$ are pairwise disjoint. Denote by $X$ the closure of the set $\{b_n:0\le n<N\}$ in $\mathbb R\cup\{\infty\}$.
\begin{enumerate}
\item The product $k_O$ converges uniformly on compact subsets of $\mathbb{C}^+$, and it defines a function in $\mathcal H$.
\item The function $k_O$ continues analytically to  $(\mathbb{C}\cup\{\infty\})\setminus X$.
\item The function $k_O|( \mathbb{R}\cup\{\infty\})\setminus X$ is real-valued.  More precisely, $k_O(x)<0$ for $x\in O$ and $k_O(x)>0$ for $x\in(\mathbb{R}\cup\{\infty\})\setminus\overline{O}$.
\item If $a_n\notin X$ for some $n$, then $a_n$ is a simple zero of $k_O$.
\item If $b_n$ is an isolated point of $X$ for some $n$, then $b_n$ is a simple pole of $k_O$.
\item Let $O_{1},O_{2}\subset\mathbb{R}\cup\{\infty\}$ be two open sets such that the
symmetric difference $O_{1}\triangle O_{2}$ has linear Lebesgue measure
equal to zero. Then $k_{O_{1}}=k_{O_{2}}$.
\item If $\varphi$ is a
conformal automorphism of $\mathbb{C}^{+}$, then $k_{\varphi^{-1}(O)}=ck_{O}\circ\varphi$,
where $c=1/|k_{O}(\varphi(i))|>0$.
\end{enumerate}
\end{prop}\begin{proof}
The convergence in (1) only needs to be discussed when $N=\infty$.  
For $z\in\mathbb{C}^{+}$,
the argument $\arg p_{J_{n}}(z)$ is equal to the angle at $z$ subtended
by the interval $(b_n,a_n)$, and therefore the sum of these arguments converges
to a number $\le\pi$. Since $|p_{(b_n,a_n)}|=1$, we conclude that the infinite product $k_O(i)$ converges and $|k_O(i)|=1$.
Convergence
elsewhere in $\mathbb{C}^{+}$ can be deduced via a normal family
argument. Indeed, the partial products $p_{(b_0,a_0)}p_{(b_1,a_1)}\cdots p_{(b_n,a_n)}$
form a normal family in $\mathbb{C}^{+}$, and any limit point $f$
of this family is such that $|f(i)|=1$ and $\arg f(z)$ equals the
angle at $z$ subtended by the set $O$. This determines the
function $f$ completely, hence $f$ is the limit (uniform on compact
subsets of $\mathbb{C}^{+}$) of these partial products.

To verify (2) and (3), we use the alternative formula $k_{O}=e^{v}$, where
\[
v(z)=\int_{O}\frac{1+tz}{t-z}\cdot\frac{dt}{1+t^{2}}.
\]
Indeed, $\Re v(i)=0$, while
\[
\Im v(z)=\Im z\int_{O}\frac{dt}{|t-z|^{2}},\quad z\notin\overline{O}.
\]
When $\Im z>0$, this integral equals precisely the angle at $z$
subtended by $O$. The function $v$ is  analytic
on $(\mathbb{C}\cup\{\infty\})\setminus\overline{O}$, and it is real-valued on $(\mathbb{R}\cup\{\infty\})\setminus\overline{O}$.  Thus $k_O(x)$ is positive on this last set. The same argument shows that $k_O/p_{(a_n,b_n)}$ can be continued analytically across the interval $(b_n,a_n)$, and the continuation is positive on this interval.  We deduce that $k_O$ can also be continued across $(b_n,a_n)$, and the continuation is negative on this interval.

Properties (4) and (5) follow from the fact that $k_O/p_(b_n,a_n)$ is analytic and positive at $a_n$ if $a_n\notin X$, and it is analytic and positive at $b_n$ if $b_n\notin X$.

To prove (6), observe that $\arg k_{O_{1}}(z)=\arg k_{O_{2}}(z)$
for $z\in\mathbb{C}^{+}$. It follows that $k_{O_{1}}/k_{O_{2}}$
is constant, and the constant is 1, as can be seen by evaluating the
quotient at $z=i$.

For (7), observe that $p_{\varphi^{-1}(J)}/p_{J}\circ\varphi$
is a positive constant (depending on $J$) for every interval $J$.
It follows that $k_{\varphi^{-1}(J)}/k_{J}\circ\varphi$ is a positive
constant as well.
\end{proof}
Part (6) of the preceding proposition indicates how a Kre\u\i n
product can be modified to minimize the number of factors. To make this precise, we define the \emph{Lebesgue regularization} of an open set $O$ to consist of
all points $x\in\mbox{\ensuremath{\mathbb{R}}}$ for which there is
$\varepsilon>0$ such that $(x-\varepsilon,x+\varepsilon)\setminus O$
has Lebesgue measure equal to zero.  Thus, the Lebesgue regularization of $O$ is the largest open set $O_1\supset O$ such that the Lebesgue measure of $O_1\setminus O$ is equal to zero, and therefore we have $k_O=k_{O_1}$. For instance,
\[
k_{(1,2)\cup(2,3)}=p_{(1,3)}.
\]
More generally, the Lebesgue regularization of a union $\bigcup_{n}(b_n,a_n)$ of subintervals of $(b,a)$ is equal to $(b,a)$ if $a-b=\sum_{n}(a_n-b_n)$. Thus, $k_{\mathbb{R}\setminus C}=-1$ if $C$ denotes the usual Cantor
ternary set.

An open set $O$ is said to be \emph{Lebesgue regular} if it equals its Lebesgue regularization.

\section{The Nevanlinna Representation}

The additive structure of the class $\mathcal{H}$ is described by
the Nevanlinna representation. Given $\alpha\ge0,\beta\in\mathbb{R}$,
and a finite, positive Borel measure $\rho$ on $\mathbb{R}$, the
function
\begin{equation}
f(z)=\alpha z+\beta+\int_{\mathbb{R}}\frac{1+zt}{t-z}\, d\rho(t),\quad z\in\mathbb{C}^{+},\label{eq:nev-formule}
\end{equation}
belongs to $\mathcal{H}$. Conversely, every function $f\in\mathcal{H}$
can be represented under this form. The constants $\alpha,\beta$
and the measure $\rho$ are uniquely determined by the function $f$.
For instance,
\[
\alpha=\lim_{y\to+\infty}\frac{f(iy)}{iy},
\]
while $\rho$ is the weak limit as $\varepsilon\downarrow0$ of the
measures
\[
d\rho_{\varepsilon}(t)=\frac{\Im f(t+i\varepsilon)}{\pi(1+t^{2})}\, dt.
\]
We denote by $\sigma(f)$ the closed support of the measure $\rho$,
to which we add $\infty$ if $\alpha>0$ or the support is unbounded. If the set $\Omega(f)=(\mathbb{R}\cup\{\infty\})\setminus\sigma(f)$
is not empty, then the formula (\ref{eq:nev-formule}) defines an
analytic function in $(\mathbb{C}\cup\{\infty\})\setminus\sigma(f)$
which takes real values on $\Omega(f)$. Conversely, if $J\subset\mathbb{R}$
is an interval such that $f$ can be extended to a continuous function
on $J$, then $J\subset\Omega(f)$. The isolated points of $\sigma(f)$
are simple poles of $f$. Observe that
\[f'(z)=\alpha+\int_{\mathbb{R}}\frac{1+t^2}{(t-z)^2}\,d\rho(t),\]
and this shows that $f$ is an increasing function on any interval in $\mathbb R$ disjoint from $\sigma(f)$.  In particular, the zeros of $f$ in such an interval must be simple.

A few simple illustrations will help clarify these notions. The function
$f(z)=z+i$ belongs to $\mathcal{H}$, and it is the restriction to
$\mathbb{C}^{+}$ of an entire function. However, this entire function
is not real-valued at any point of $\mathbb{R}$, and therefore $\sigma(f)=\mathbb{R}\cup\{\infty\}.$
The measure $\rho$ corresponding to this function is the Cauchy distribution
\[
d\rho(t)=\frac{dt}{\pi(1+t^{2})}.
\]
A more interesting example is the function
\[
f(z)=z+\sqrt{z^{2}-1},
\]
where the square root is taken to be positive for $z>1$, for which
we have $\sigma(f)=[-1,1]\cup\{\infty\}$. For this example we have
\[
-\frac{1}{f(z)}=\sqrt{z^{2}-1}-z,
\]
and $\sigma(-1/f)=[-1,1]$.  For a Kre\u\i n factor $p_J$, where $J=(b,a)$ is a proper interval, we have $\sigma(p_J)=\{b\}$.

For any function $f\in\mathcal{H}$, the limit
\[
f(t)=\lim_{\varepsilon\downarrow0}f(t+i\varepsilon)
\]
exists for almost every $t\in\mathbb{R}$ (relative to Lebesgue measure).
Moreover, the absolutely continuous part of the measure $\rho$ in
the representation (\ref{eq:nev-formule}) is equal to
\[
\frac{\Im f(t)\, dt}{\pi(1+t^{2})}.
\]
In particular, the measure $\rho$ is singular relative to Lebesgue
measure if and only if $f(x)$ is real for almost every $x\in\mathbb{R}$. 

The following result was proved by G. Letac \cite{letac-proc}; the
case of measures with finite support  is due to Boole \cite{boole}.
\begin{thm}
\label{thm:measure-preserving}Assume that the function $f\in\mathcal{H}$
is given by 
\[
f(z)=z+\beta+\int_{\mathbb{R}}\frac{1+zt}{t-z}\, d\rho(t),\quad z\in\mathbb{C}^{+},
\]
where $\rho$ is singular relative to Lebesgue measure. Then the map
$f$ preserves Lebesgue measure, that is, the Lebesgue
measure of the set 
\[
\{x\in\mathbb{R}:f(x)\in(c,d)\}
\]
equals $d-c$ when $-\infty<c<d<\infty$.
\end{thm}
We will actually require an equivalent result, first proved by Hru\v s\v c\"ev
and Vinogradov \cite{hru-vin}. We include the simple derivation from
Theorem \ref{thm:measure-preserving}. Given a finite, positive, measure
$\mu$ on $\mathbb{R}$, singular with respect to Lebesgue measure,
the Cauchy transform
\[
G_{\mu}(z)=\int_{\mathbb{R}}\frac{d\mu(t)}{z-t},\quad z\in\mathbb{C}^{+},
\]
has real limit values
\[
G_{\mu}(x)=\lim_{\varepsilon\downarrow0}G_{\mu}(x+i\varepsilon)
\]
for almost every $x\in\mathbb{R}$. Here is the result of \cite{hru-vin};
the special case of measures with finite support is due again to Boole
\cite{boole}.
\begin{thm}
\label{thm:hruscev-vino}Let $\mu$ be a finite, positive measure
on $\mathbb{R}$, singular with respect to Lebesgue measure. For every
$y>0$, the Lebesgue measure of the sets
\[
\{x\in\mathbb{R}\colon G_{\mu}(x)>y\},\quad\{x\in\mathbb{R}\colon G_{\mu}(x)<-y\}
\]
is equal to $\mu(\mathbb{R})/y$.\end{thm}
\begin{proof}
It suffices to consider the case when $\mu$ is a probability measure.
The function $f(z)=1/G_{\mu}(z)$ is easily seen to belong to $\mathcal{H}$,
and $f(x)$ is real for almost every $x\in\mathbb R$. Moreover, since 
\[
\lim_{y\to\infty}iyG_{\mu}(iy)=\mu(\mathbb{R})=1,
\]
we deduce that Theorem \ref{thm:measure-preserving} applies to $f$.
The result follows now immediately because, for instance,
\[
\{x\in\mathbb{R}\colon G_{\mu}(x)>y\}=\{x\in\mathbb{R}:f(x)\in(0,1/y)\}.\qedhere
\]
\end{proof}
We need an easy consequence of this result.
\begin{cor}
\label{cor:unbounded_f}Assume that the function $f\in\mathcal{H}$
is given by \emph{(\ref{eq:nev-formule}),} and the restriction of
the measure $\rho$ to an interval $J$ is nonzero but singular relative
to Lebesgue measure. Then the function $f$ is not essentially
bounded below or above on the interval $J$.\end{cor}
\begin{proof}
Denote by $\mu$ the restriction of the measure $(1+t^{2})\, d\rho(t)$
to the interval $J$. Then the function $f(z)+G_{\mu}(z)$ can be
continued analytically across the interval $J$, thus $f(x)+G_{\mu}(x)$
is bounded on any compact subset of $J$. The corollary follows now
immediately from Theorem \ref{thm:hruscev-vino}.\end{proof}
\begin{cor}
\label{cor:lebesgur-regularity-of_f>0}For every $f\in\mathcal{H}$,
the set
\[
\Gamma(f)=\{x\in\Omega(f):f(x)<0\}
\]
is Lebesgue regular.\end{cor}
\begin{proof}
Assume to the contrary that there exists $x\in\mathbb{R}\setminus\Gamma(f)$
and $\varepsilon>0$ such that the interval $J=(x-\varepsilon,x+\varepsilon)$
is contained almost everywhere in $\Gamma(f)$, but $J\cap\sigma(f)\ne\varnothing.$
If the function $f\in\mathcal{H}$ is given by (\ref{eq:nev-formule}),
it follows that the restriction of $\rho$ to $J$ is nonzero, and
this restriction is singular as well since its support is a closed
set of Lebesgue measure equal to zero. Corollary \ref{cor:unbounded_f}
implies that $f$ is not essentially bounded above on $J$, contrary
to the assumption that $f<0$ almost everywhere on $J$.
\end{proof}

Proposition \ref{prop:krein-prod-properties} describes the sets
 $\sigma(f),\Omega(f),$ and $\Gamma(f)$ in
the case of a Kre\u\i n product.  We state the result for further use.

\begin{prop}
Consider a Lebesgue-regular open set $O=\bigcup_{0\le n<N}(b_n,a_n)\subset\mathbb R\cup\{\infty\}$, where the intervals $(b_n,a_n)$ are pairwise disjoint. Denote by $X$ the closure of the set $\{b_n:0\le n<N\}$ in $\mathbb R\cup\{\infty\}$. Then
\begin{enumerate}
\item $\sigma(k_O)=X$;
\item $\Gamma(k_O)=O$;
\item the real poles of $k_O$ are precisely the points $b_n$ which are isolated in $X$;
\item a point $a_n$ is in $\sigma(k_O)$ precisely when $a_n\in X$, otherwise $a_n$ is a simple zero.
\end{enumerate}
\end{prop}

\section{Factorization}

We begin by considering factors of the form $p_{J}$.
\begin{lem}
\label{lem:factor_a_single_p_J}Consider a nonzero function $f\in\mathcal{H}$,
and an interval $J\subset\Omega(f)$ such that $f(x)<0$ for all $x\in J$.
Then there exists a function $g\in\mathcal{H}$ such that $f=p_{J}g$.\end{lem}
\begin{proof}
Replacing the function $f$ by $f\circ\varphi$ for some conformal
automorphism of $\mathbb{C}^{+}$, we may assume that $J=(-\infty,0)$,
so that we need to prove that the function $f(z)/z$ belongs to $\mathcal{H}$.
Since 
\[
\frac{f(z)}{z}=\lim_{\varepsilon\downarrow0}\frac{f(z-\varepsilon)}{z},\quad z\in\mathbb{C}^{+},
\]
and $\mathcal{H}$ is closed under pointwise limits, we may also assume
that $\sigma(f)\subset[\varepsilon,+\infty]$ for some $\varepsilon>0$,
and $f(0)<0$. Thus $f$ can be written as
\[
f(z)=\alpha z+\beta+\int_{[\varepsilon,+\infty)}\frac{1+zt}{t-z}\, d\rho(t),\quad z\in\mathbb{C}^{+},
\]
for some $\alpha>0,\beta\in\mathbb{R}$, and some finite, positive,
Borel measure $\rho$ on $[\varepsilon,+\infty).$ We have then
\[
f(0)=\beta+\int_{[\varepsilon,+\infty)}\frac{1}{t}\, d\rho(t)<0,
\]
so that
\begin{eqnarray*}
\frac{f(z)}{z} & = & \alpha+\left[\beta+\int_{[\varepsilon,+\infty)}\frac{1}{t}\, d\rho(t)\right]\frac{1}{z}+\int_{[\varepsilon,+\infty)}\left[\frac{1+zt}{z(t-z)}-\frac{1}{zt}\right]\, d\rho(t)\\
 & = & \alpha+\left[\beta+\int_{[\varepsilon,+\infty)}\frac{1}{t}\, d\rho(t)\right]\frac{1}{z}+\int_{[\varepsilon,+\infty)}\frac{1+t^{2}}{t(t-z)}\, d\rho(t),
\end{eqnarray*}
and it is clear that this function has nonnegative imaginary part
for $z\in\mathbb{C}^{+}$.
\end{proof}
Our main factorization result follows. Recall that $\Omega(f)=(\mathbb{R}\cup\{\infty\})\setminus\sigma(f)$
and $\Gamma(f)=\{x\in\Omega(f)\colon f(x)<0\}$.
\begin{thm}
\label{thm:main-factorization}For every nonzero function $f\in\mathcal{H}$
there exists $g\in\mathcal{H}$ such that $f=k_{\Gamma(f)}g$. The
function $g$ has the following properties.
\begin{enumerate}
\item $\sigma(g)\subset\sigma(f)$.
\item $g(x)>0$ for every $x\in\Omega(g)$.
\item The set $\Omega(g)$ is Lebesgue regular.
\item $\Omega(f)=\Omega(k_{\Gamma(f)})\cap\Omega(g)$.
\end{enumerate}
\end{thm}
\begin{proof}
Denote by $(J_{n})_{0\le n<N}$ the connected components of $\Gamma(f)$,
so that 
\[
k_{\Gamma(f)}=\prod_{0\le n<N}p_{J_{n}}.
\]
 We first verify by induction that the function
\[
g_{k}=f/\prod_{0\le n<k}p_{J_{n}}
\]
belongs to $\mathcal{H}$ for every finite $k\le N$. The case $k=0$
is vacuously verified since $g_{0}=f$. The function $g_{k}$ has
the property that $J_{k+1}\subset\Gamma(g_{k})$, and thus the fact
that $g_{k+1}\in\mathcal{H}$ follows by an appplication of Lemma
\ref{lem:factor_a_single_p_J}. If $N$ is finite, we are done showing
that $g\in\mathcal{H}$. If $N$ is infinite, we have $g(z)=\lim_{n\to\infty}g_{n}(z)$
for all $z\in\mathbb{C}^{+}$, and we conclude again that $g\in\mathcal{H}$.

Assume next that $x\in\Omega(f)$. If $f(x)>0$ then $x$ is at positive
distance from $\Gamma(f)$, and therefore the product $k_{\Gamma(f)}$
converges on a neighborhood of $x$, and it takes positive values
at real points close to $x$. (This argument also works for $x=\infty$
with the usual interpretation of the word `close'.) It follows that
$x\in\Omega(g)$. If $f(x)<0$ then $k_{\Gamma(f)}$ is again analytic
in a neighborhood of $x$, and it takes negative values at real points
close to $x$. We conclude again that $x\in\Omega(g)$. If $f(x)=0$,
it follows that, for some $\varepsilon>0$, $(x-\varepsilon,x)\subset\Gamma(f)$
and $(x,x+\varepsilon)\cap\Gamma(f)=\varnothing$. In particular,
$f$ has a simple zero at $x$. In this case $k_{\Gamma(f)}$ also
has a simple zero at $x$, and the ratio $g=f/k_{\Gamma(f)}$ is positive
in $(x-\varepsilon,x+\varepsilon)$ and analytic near $x$. In particular,
$x\in\Omega(g)$. This verifies property (1).

The above argument also shows that $g(x)>0$ for every $x\in\Omega(f)$.
Assume now that $x\in\Omega(g)\setminus\Omega(f)$ and $g(x)<0$.
Since $k_{\Gamma(f)}=f/g$, we deduce that we have $x\in\sigma(k_{\Gamma(f)})$,
and therefore $x$ is an accumulation point of a sequence of endpoints
of some intervals $J_{n}$. Since $g$ is positive on each $J_{n}$,
we deduce $g(x)\ge0,$ a contradiction. Thus we must have $g(x)\ge0$
for $x\in\Omega(g)$. However, when a function $g$ in $\mathcal{H}$
vanishes at some point $x\in\Omega(g)$, the function $g$ must change
sign in a neighborhood of that point. This proves (2).

Property (3) follows from Corollary \ref{cor:lebesgur-regularity-of_f>0}
applied to the function $-1/g$.

The inclusion  $\Omega(f)\supset\Omega(k_{\Gamma(f)})\cap\Omega(g)$ is obvious.  Conversely, observe that any point $a\in\Omega(f)$ such that $f(a)=0$ is a simple zero of $f$ and of $k_{\Gamma(f)}$.  Indeed, there is an interval $(b,a)\subset\Gamma(f)$, and $a$ is isolated in $\partial\Gamma(f)$.  Moreover, $k_{\Gamma(f)}$ is  analytic and real in a neighborhood of $a$.  It follows that the quotient $g=f/k_{\Gamma(f)}$ is analytic and real on a neighborhood of $a$.  It is also clear that $f$ and $k_{\Gamma(f)}$ are analytic and nonzero in the neighborhood of any point  $a\in\Omega(f)$ such that $f(a)\ne0$. Thus both $k_{\Gamma(f)}$ and $g$ are analytic and real on $\Omega(f)$, thus verifying the opposite inclusion.
\end{proof}
The preceding result is most effective when $\sigma(f)$ is small.
\begin{cor}
\label{cor:sigma-of-measure-zero}Assume that $f\in\mathcal{H}$ and
$\sigma(f)$ has Lebesgue measure equal to zero. Then we have $f=ck_{\Gamma(f)}$
for some positive constant $c$.\end{cor}
\begin{proof}
Consider the factorization $f=k_{\Gamma(f)}g$ provided by the preceding
theorem. The set $\Omega(g)$ is Lebesgue regular, and hence $\Omega(g)=\mathbb{R}\cup\{\infty\}$
because its complement has measure zero. The Nevanlinna representation
of $g$ shows then that $g$ is a constant function, and the constant
is positive by Theorem \ref{thm:main-factorization}(2).
\end{proof}
The preceding corollary recovers Kre\u\i n's original factorization
result, as well as the extensions considered in \cite{chal-et-al}. Note that
the measure $\rho$ in the Nevanlinna representations of the functions
considered in \cite{chal-et-al} is a discrete measure whose support has only
finitely many accumulation points. Our corollary covers functions
for which the support of $\rho$ is an arbitrary closed set of Lebesgue
measure zero, for instance the ternary Cantor set.

\section{Functions $g$ Positive on $\Omega(g)$}

In this section we analyze more carefully the factor $g$ in the decomposition
$f=k_{\Gamma(f)}g$. The characteristic properties of these functions
are that $\Omega(g)$ is Lebesgue regular, and $g(x)>0$ for $x\in\Omega(g)$.
Such a function can then be written as $g=e^{h}$ for some function
$f\in\mathcal{H}$ such that 
\[
0<\Im h(z)<\pi
\]
for all $z\in\mathbb{C}^{+}$. Moreover, the function $h$ can be
continued analytically across $\Omega(g)$, and $h(x)\in\mathbb{R}$
for $x\in\Omega(g)$. Conversely, $g$ can be continued analytically
across $\Omega(h)$, and $g(x)>0$ for $x\in\Omega(h)$. We conclude
that $\Omega(h)=\Omega(g)$ and $\sigma(h)=\sigma(g)$. These facts
imply that the Nevanlinna representation of the function $h$ is of
the form
\[
h(z)=\gamma+\int_{\sigma(g)}\frac{1+zt}{t-z}\cdot\frac{\psi(t)}{1+t^{2}}\, dt,\quad z\in\mathbb{C}^{+},
\]
where $\gamma\in\mathbb R$, $\psi:\sigma(g)\to(0,1]$ is a measurable function, and there is no open interval $J$ such that $\psi(t)=1$ almost everywhere on $J$. The fact that $\sigma(h)=\sigma(g)$ amounts
to saying that the support of the measure $\psi(t)\, dt$ is equal
to $\sigma(g)$ or, equivalently since $\Omega(g)$ is Lebesgue regular,
$\psi(t)>0$ for almost every $t\in\sigma(g)$.

Conversely, assume that $\sigma\subset\mathbb{R}$ is a closed set
such that $\mathbb{R}\setminus\sigma$ is Lebesgue regular, and $\varphi:\sigma\to(0,1]$
is measurable. We can then define a function $v\in\mathcal{H}$ by
setting
\[
v(z)=\int_{\sigma}\frac{1+zt}{t-z}\cdot\frac{\varphi(t)}{1+t^{2}}\, dt,\quad z\in\mathbb{C}^{+}.
\]
 This function satisfies
\[
\Im v(z)=\Im z\int_{\sigma}\frac{\varphi(t)}{|t-z|^{2}}\, dt\le\Im z\int_{\sigma}\frac{dt}{|t-z|^{2}}\le\Im z\int_{-\infty}^{\infty}\frac{dt}{|t-z|^{2}}=\pi.
\]
Moreover, $\Gamma(g)=\varnothing$ provided that $\varphi$ is not equal to $1$ almost everywhere on any open interval.

We summarize these observations below.
\begin{thm}
\label{thm:additional-info-about-g}Assume that $f\in\mathcal{H}$.
There exist a constant $\gamma\in\mathbb{R}$ and a measurable function
$\psi:\sigma(f)\to[0,1]$ such that
\[
f(z)=k_{\Gamma(f)}(z)e^{h(z)},\quad z\in\mathbb{C}^{+},
\]
 where $h\in\mathcal{H}$ is defined by
\[
h(z)=\gamma+\int_{\sigma(f)}\frac{1+zt}{t-z}\cdot\frac{\psi(t)}{1+t^{2}}\, dt,\quad z\in\mathbb{C}^{+}.
\]
Conversely, given a closed set $\sigma\subset\mathbb{R}$, an open
subset $O\subset(\mathbb{R}\cup\{\infty\})\setminus\sigma$, and a
measurable function $\varphi:\sigma\to[0,1]$ such that $\phi$ is not equal to $1$ almost everywhere on any open interval, the function
\[
k_{O}(z)e^{v(z)},\quad z\in\mathbb{C}^{+},
\]
where $v$ is defined by
\[
v(z)=\int_{\sigma}\frac{1+zt}{t-z}\cdot\frac{\varphi(t)}{1+t^{2}}\, dt,\quad z\in\mathbb{C}^{+},
\]
belongs to the class $\mathcal{H}$, and $e^{v(t)}>0$ for every $t\in\Gamma(e^v)$.\end{thm}
\begin{proof}
We only need to verify the last assertion. Recall that the argument
of $k_{O}(z)$ equals the angle subtended by the set $\Gamma$ at
$z$, that is
\[
\Im z\int_{O}\frac{dt}{|t-z|^{2}}.
\]
 As shown above, the argment of $e^{v(z)}$ is at most 
\[
\Im z\int_{\sigma}\frac{dt}{|t-z|^{2}},
\]
and the sum of these two numbers is at most $\pi$ because $\sigma\cap O=\varnothing$.
It follows that the argument of $k_{\Gamma}e^{v}$ is at most equal
to $\pi$, and therefore this function belongs to $\mathcal{H}$.
\end{proof}

\section{Interpolation}

A natural question arises as to which pairs $(\Omega,O)$ of open
subsets of $\mathbb{R}\cup\{\infty\}$ are of the form $(\Omega(f),\Gamma(f))$.

\begin{prop}
\label{prop:which-omega-and-gamma-occur}Consider open sets $\Omega$
and $O=\bigcup_{0\le n<N}(b_n,a_n)$ in $\mathbb{R}\cup\{\infty\}$. Denote by $X$ the closure in $\mathbb{R}\cup\{\infty\}$ of the set $\{b_n:0\le n<N\}$, and by $\Omega_1$ the Lebesgue regularization of $\Omega$. The following conditions are equivalent.
\begin{enumerate}
\item There exists a function $f\in\mathcal{H}$ such that $\Omega(f)=\Omega$
and $\Gamma(f)=O$.
\item The sets $\Omega$ and $O$ satisfy the following three requirements:

\begin{enumerate}
\item $O\subset\Omega$;
\item $O$ is Lebesgue regular;
\item $\Omega=\Omega_1\setminus X$.
\end{enumerate}
\end{enumerate}
\end{prop}
\begin{proof}
Assume first $f\in\mathcal{H}$ satisfies (1), and factor $f=k_Og$ for some $g\in\mathcal{H}$.  Condition (a) is obviously satisfied.  The sets $O$ and $\Omega(g)$  are Lebesgue regular by Corollary \ref{cor:lebesgur-regularity-of_f>0} and Theorem \ref{thm:main-factorization}(3), and
\[
\Omega=\Omega(f)=\Omega(g)\cap\Omega(k_O)=\Omega(g)\setminus X.
\]
Condition (c) now follows because $\Omega(g)\supset\Omega_1$ by regularity, while $\Omega\cap X=\varnothing$. 

Conversely, assume that conditions (a--c) are satisfied, and set $f=k_Og$, where $g=e^v$ and $v$ is defined as
\[v(z)=\frac12\int_{\mathbb{R}\setminus\Omega_1}\frac{1-zt}{t-z}\cdot\frac{dt}{1+t^2}\quad z\in\mathbb{C}^+.\]
We have $\Gamma(f)=O$ and $\Omega(g)=\Omega_1$, while $\Omega(k_O)=\mathbb R\setminus X$.  It is now easy to conclude using Theorem \ref{thm:main-factorization}(4) that \[\Omega(f)=\Omega(g)\cap\Omega(k_O)=\Omega_1\cap\Omega(k_O)=\Omega_1\setminus X=\Omega.\]
We conclude that $f$ satisfies (1).
\end{proof}
It is easy to construct examples of sets $\Omega$ and $O$ satisfying conditions (a-c) of the preceding theorem.  Consider, for instance the ternary Cantor set $C$, obtained by removing $2^{k-1}$ intervals of length $3^{-k}$ from the interval $[0,1]$ for $k\ge1$, and set $\Omega=\mathbb{R}\setminus C$.  Denote by $(b_n,c_n)$ these intervals, and select for each $n$ a point $a_n\in(b_n,c_n)$.  Then $O=\bigcup_n(b_n,a_n)$ is a Lebesgue regular open set and $k_O$ has zeros at the points $a_n$ and essential singularities at all the points of $C$.  If we consider instead the union $O$ of all intervals $(b_n,c_n)$ of length $3^{-k}$ with $k$ even, then $k_O$ has no zeros or poles and it is still singular at all points in $C$.  In general, if $\Omega=\Omega(f),O=\Gamma(f)$,  and $(b,c)$ is a connected component of $\Omega$, the intersection $(b,c)\cap O$ can be any interval of the form $(b,a)$ for some $a\in[b,c]$.

The following result can be viewed as an interpolation result which yields, in the special case of a finite set $Y$, the results of \cite{chal-et-al}.  The set $A$ in the statement is the proposed set of zeros of a function in $\mathcal H$, $B$ is the proposed set of poles, and $Y$ is a set where the function is allowed to be essentially singular.  The conditions in assertion (3) could be summarized by saying that the points of $A$ and $B$ are \emph{interlaced} in each component of the complement of $Y$.

\begin{thm}
Consider pairwise disjoint sets $A,B,Y\subset\mathbb R$ such that $Y$ is closed, and the limit points of $A\cup B$ belong to $Y$.  The following conditions are equivalent.
\begin{enumerate}
\item There exists a function $f\in\mathcal{H}$ such that $A\subset\{x\in\Omega(f):f(x)=0\}\subset A\cup Y$ and $B\subset\sigma(f)\subset B\cup Y$.
\item There exists a Lebesgue regular open set $O=\bigcup_{0\le n<N}(b_n,a_n)$ such that the intervals $(b_n,a_n)$ are pairwise disjoint, $A\subset\{a_n:0\le n<N\}\subset A\cup Y$, and $B\subset\{b_n:0\le n<N\}\subset B\cup Y$.
\item For every connected component $J$ of $(\mathbb{R}\cup\{\infty\})\setminus Y$, the following two conditions are satisfied.\begin{enumerate}
\item If $a$ and $a'$ are two distinct points in $A\cap J$, there is a point $b\in B$ between $a$ and $a'$.
\item If $b$ and $b'$ are two distinct points in $B\cap J$, there is a point $a\in A$ between $b$ and $b'$.
\end{enumerate}
\end{enumerate}
If these equivalent conditions are satisfied and $Y$ is of Lebesgue measure equal to zero, then every function $f$ satisfying \emph{(1)} is of the form $f=ck_O$, where $c>0$ and $O$ is one of the open sets satisfying \emph{(2)}.
\end{thm}
\begin{proof}
Assume first that $f$ satisfies the conditions in (1),  set $O=\Gamma(f)$, and write $O=\bigcup_{0\le n<N}(b_n,a_n)$ with pairwise disjoint intervals $(b_n,a_n)$.  Then $f$ has a zero at each point $a\in A$, and therefore $a=a_n$ for some $n$.  On the other hand, if $a\notin A\cup Y$ then $f$ is (at worst) meromorphic in a neighborhood of $a$ and $f(a)\ne0$, so that $a\ne a_n$ for all $n$.  Similarly, if $b\in B$ then $b$ is isolated in $B\cup Y$, in particular $b$ is isolated in $\sigma(f)$.  It follows that $b$ is a pole of $f$, and therefore $b=b_n$ for some $n$.  Finally, if $b\notin B\cup Y$ then $b\notin\sigma(f)$ so that $f$ is analytic at $b$.  We conclude that $b\ne b_n$ for all $n$.  We have then verified that $O$ satisfies the conditions in (2).

Conversely, assume that $O$ is a Lebesgue regular open set satisfying the conditions in (2), and set $f=k_O$.  The set $\sigma(f)$ is the closure of $\{b_n:0\le n<N\}$ and it therefore contained in $B\cup Y$.  Every $a\in A$ is of the form $a=a_n$ for some $n$, and the hypothesis on the set $A$ implies that $a_n\notin B\cup Y$.  We conclude that $a_n\in\Omega(f)$ and $f(a_n)=0$.  On the other hand, if $a\notin A\cup Y$ then $a\ne a_n$ for all $n$ and therefore $f(a)\ne0$ when $a\in\Omega(f)$.  This proves the first two inclusions in (1).  Similarly, each $b\in B$ is equal to some $b_n$ and it is an isolated point in $\sigma(f)$, thus a pole.  Points $b\notin B\cup Y$ belong to $\Omega(f)$, and this verifies the last two inclusions in (1).

We have established the equivalence of (1) and (2).  The equivalence of (2) and (3) is easily verified.  Indeed, assume that (2) is satisfied, $J$ is a component of the complement of $Y$, and $a,a'\in A\cap Y$.  There are then integers $n,m$ so that $a=a_n$ and $a'=a_m$.  Then we have $b_n,b_m\in B\cap Y$, and therefore one of these points is between $a$ and $a'$. The second condition in (3) is verified similarly.  

Conversely, assume that (3) holds.  For each connected component $(c,d)$ of the complement of $Y$ there exists then a family $\{(b_i,a_i):i\in I\}$
 of pairwise disjoint intervals contained in $(c,d)$ such that $A\cap (c,d)\supset\{a_i :i\in I\}$,  $B\cap (c,d)\supset\{b_i :i\in I\}$, and the sets $(A\cap (c,d))\setminus\{a_i:i\in I\}$ and $(B\cap (c,d))\setminus\{b_i:i\in I\}$ contain at most one point each. If $(A\cap (c,d))\setminus\{a_i:i\in I\}=\{a\}$, then $(c,a)$ is disjoint from $(b_i,a_i)$ for all $i$.  Analogously, if $(B\cap (c,d))\setminus\{b_i:i\in I\}={b}$  then $(b,d)$ is disjoint from all the intervals $(b_i,a_i)$.  The set $O$ is now defined to be the smallest Lebesgue regular open set containing all the intervals $(b_i,a_i)$ as well as the intervals $(c,a)$ and $(b,d)$ when needed.

The last assertion in the statement follows immediately from Corollary \ref{cor:sigma-of-measure-zero}.
\end{proof}

The special points $a,b$ which appear in the proof of $(3)\Rightarrow(2)$ are the `loners' of \cite{chal-et-al}.  When a loner $a$ exists, the function $f=k_O$ has a pole at $c\in Y$ unless there is also a loner on the other side of $c$.  Similarly, the function $f$ may have a zero at $d$.

When a component $(c,d)$ of $(\mathbb{R}\cup\{\infty\})\setminus Y$  contains no points in $A\cup B$, the entire interval $(c,d)$ can be added to the set $O$. This is the extent of nonuniqueness allowed in the selection of the set $O$, and thus in the choice of the function $f$, when $Y$ has zero Lebesgue measure.

The preceding result can be reformulated as an interpolation result for self maps of the unit disk using the conformal map $z\mapsto (z-\zeta)/(z+\overline{\zeta})$ of $\mathbb{C}^+$ onto the disk $\mathbb D$ for some $\zeta\in\mathbb{C}^+$.  We state a particular case which may be useful in other contexts.  We denote by $\mathbb{T}=\partial\mathbb{D}$ the unit circle.

\begin{cor}
Assume that $A,B$, and $Z$ are three pairwise disjoint subsets of $\mathbb T$ such that $Z$ is closed, the limit points of $A$ and $B$ are contained in $Z$, and $A$ and $B$ are interlaced on every component of $\mathbb{T}\setminus Z$.  Fix $\alpha,\beta\in\mathbb T$ with $\alpha\ne\beta$. Then there exists an analytic function $\vartheta:\mathbb D\to\mathbb D$ such that
\begin{enumerate}
\item $\vartheta$ can be continued analytically across $\mathbb{T}\setminus Z$;
\item $|\vartheta(z)|=1$ for all $z\in\mathbb{T}\setminus Z$;
\item $A=\{a\in\mathbb{T}\setminus Z:\vartheta(a)=\alpha\}$ and $B=\{b\in\mathbb{T}\setminus Z:\vartheta(b)=\beta\}$.
\end{enumerate}
When $Z$ has zero linear Lebesgue measure, the function $\theta$ can be chosen to be inner.
\end{cor}

The case of a finite set $Z$ is considered in \cite{chal-et-al}, while $Z=\varnothing$, corresponding to finite Blaschke products, was studied in \cite{go}.

\end{document}